\documentclass[a4paper,11pt]{amsart}

\usepackage{latexsym}
\usepackage{a4wide}
\usepackage{amscd}
\usepackage{graphics}
\usepackage{amsmath}
\usepackage{amssymb}
\usepackage{mathrsfs}
\usepackage{enumerate}
\usepackage{hyperref}
\usepackage{pgf,tikz}

\newcommand\N{\ensuremath{\mathbb{N}}}
\newcommand\Z{\ensuremath{\mathbb{Z}}}
\newcommand\C{\ensuremath{\mathbb{C}}}
\newcommand\R{\ensuremath{\mathbb{R}}}
\newcommand\RR{\ensuremath{\mathbb{R}}}
\newcommand\CC{\ensuremath{\mathbb{C}}}

\newcommand\NN{\ensuremath{\mathbb{N}}}

\newcommand{\eps}{\varepsilon}

\renewcommand{\Re}{\mathop{\mathfrak{Re}}}
\renewcommand{\Im}{\mathop{\mathfrak{Im}}}

\newtheorem{Theorem}{Theorem}[section]

\newtheorem{Lemma}[Theorem]{Lemma}
\newtheorem{Proposition}[Theorem]{Proposition}
 \theoremstyle{definition} \newtheorem{Definition}[Theorem]{Definition}

 \newtheorem{remark}[Theorem]{Remark}

\begin{document}

\title[Aharonov-Bohm operators with two colliding poles]{On {A}haronov-{B}ohm
  operators with two colliding poles}

\author{Laura Abatangelo}

\address{
\hbox{\parbox{5.7in}{\medskip\noindent
  L. Abatangelo\\
Dipartimento di Matematica e Applicazioni,
 Universit\`a di Milano--Bicocca, \\
Via Cozzi 55, 20125 Milano (Italy). 
         {\em{E-mail address: }}{\tt laura.abatangelo@unimib.it.}}}
}

\author{Veronica Felli}

\address{
\hbox{\parbox{5.7in}{\medskip\noindent
   V. Felli\\
Dipartimento di Scienza dei Materiali,
 Universit\`a di Milano--Bicocca, \\
Via Cozzi 55, 20125 Milano (Italy). 
         {\em{E-mail address: }}{\tt veronica.felli@unimib.it.}}}
}

\author{Corentin L\'ena }

\address{
\hbox{\parbox{5.7in}{\medskip\noindent
  C. L\'ena\\
Dipartimento di Matematica \emph{Giuseppe Peano},
 Universit\`a di Torino, \\
Via Carlo Alberto 10, 10123 Torino (Italy). 
         {\em{E-mail address: }}{\tt clena@unito.it.}}}
}

\date{\today}

\thanks{ The authors have been partially supported by the project ERC
  Advanced Grant 2013 n. 339958: ``Complex Patterns for Strongly
  Interacting Dynamical Systems --
  COMPAT''. V. Felli is partially supported by PRIN-2012-grant ``Variational and perturbative aspects of nonlinear differential problems''.
  }

\begin{abstract}
  We consider Aharonov-Bohm operators with two poles and prove sharp
  asymptotics for simple eigenvalues as the poles collapse at an
  interior point out of nodal lines of the limit eigenfunction.
\end{abstract}

\maketitle

\section{Introduction}

The present paper is concerned with asymptotic estimates of the
eigenvalue variation for magnetic Schr\"odinger operators with Aharonov--Bohm potentials.
These special potentials generate
localized magnetic fields, as they are produced by
infinitely long thin solenoids intersecting perpendicularly the plane
at fixed points (poles), as the radius of the solenoids goes to
zero and the magnetic flux remains constant.

The aim of the present paper is the investigation of
eigenvalues of these operators as functions of the poles on the
domain.  This study was initiated by the set of papers
\cite{AF,abatangelo2016leading,abatangelo2016boundary,BNNNT,NNT},
where the authors consider a single point moving in the domain,
providing sharp asymptotics as it goes to an interior point or to a
boundary point.  On the other hand, to the best of our knowledge the
only paper considering different poles is \cite{lena}, providing a
continuity result for the eigenvalues and an improved regularity for
simple eigenvalues as the poles are distinct and far from the
boundary.

Additional motivations for the study of eigenvalue functions of these
operators appear in the theory of spectral minimal partitions.  We
refer the interested reader to \cite{BNH2011,BNHHO2009,HHO13,NT} and
references therein.

For $a = (a_1, a_2) \in \RR^2$,
the Aharonov-Bohm magnetic potential with pole
$a$ and circulation $1/2$ is defined as 
\[
{\mathbf A}_a(x) = \frac12 \left( \frac{ - (x_2 - a_2)}{(x_1 - a_1)^2 + (x_2 -
a_2)^2} , \frac{x_1 - a_1}{(x_1 - a_1)^2 + (x_2 - a_2)^2} \right), 
\quad x=(x_1,x_2) \in \RR^2 \setminus \{a\}.
\]
In this paper we consider
potentials which are the sum of two different Aharonov--Bohm potentials
whose singularities 
are located at two different points in the domain moving towards each other. 
For $a>0$ small, let $a^-=(-a,0)$ and $a^+=(a,0)$ be the poles of the following Aharonov--Bohm potential
\begin{equation}\label{eq:11}
	{\mathbf A}_{a^-\!,a^+}(x):=
-{\mathbf A}_{a^-}+{\mathbf A}_{a^+}
= -\frac12\frac{(-x_2,x_1+a)}{(x_1+a)^2+x_2^2}+\frac12\frac{(-x_2,x_1-a)}{(x_1-a)^2+x_2^2}.
\end{equation}
 Let $\Omega$ be 
an open, bounded, and connected set in $\RR^2$ such that
$0\in\Omega$. 
We consider the Schr\"odinger operator 
\begin{equation}\label{eq:12}
H_{a^-\!,a^+}^{\Omega}= (i\nabla+{\mathbf A}_{a^-\!,a^+})^2
\end{equation}
with homogenous Dirichlet boundary conditions  (see
\S \ref{sec:gauge-invariance} for the notion of magnetic Hamiltonians) and
its eigenvalues $(\lambda_k^a)_{k\ge 1}$, counted with
multiplicities.
We denote by $(\lambda_k)_{k\ge 1}$ the eigenvalues of the Dirichlet
Laplacian $-\Delta$ in $\Omega$. As already mentioned, we know from  \cite{lena} that,
for every $k\geq1$, 
\begin{equation}\label{eq:16}
\lim_{a\to 0} \lambda_k^a=\lambda_k.
\end{equation}
The main result of the present paper is a sharp
asymptotics for the eigenvalue variation $\lambda_k^a-\lambda_k$ as the
two poles $a^-,a^+$ coalesce towards a point where the limit
eigenfunction does not vanish.

A first result in this direction was given in \cite{AFHL-1}, under a symmetry assumption on the domain.
\begin{Theorem}{\cite[Theorem 1.13]{AFHL-1}}\label{t:ab}
Let $\sigma:\RR^2\to \RR^2$,  $\sigma(x_1,x_2)=(x_1,-x_2)$.
 Let $\Omega$ be 
an open, bounded, and connected set in $\RR^2$,
satisfying $\sigma(\Omega)=\Omega$ and $0\in\Omega$. 
Let $\lambda_N$ be a simple eigenvalue of the Dirichlet Laplacian on $\Omega$
and $u_N$ be a $L^2(\Omega)$-normalized eigenfunction associated
to $\lambda_N$.
Let $k\in \NN\cup\{0\}$ be the order of vanishing of $u_N$ at $0$ and
$\alpha\in[0,\pi)$ be such that the minimal slope of
nodal lines of $u_N$ is equal to $\frac{\alpha}{k}$, so that 
\begin{equation*}
  u_N(r(\cos t,\sin t)) \sim r^{k} 
  \beta \sin(\alpha-kt)\quad\text{as $r\to0^+$ for all $t$},
 \end{equation*}
for some  $\beta\in\RR\setminus\{0\}$ (see e.g. \cite{FFT}). 
Let us  assume that $\alpha\neq0$.

For $a>0$ small, let $a^-=(-a,0)$, $a^+=(a,0)\in\Omega$,  and let $\lambda_N^a$
be the $N$-th
 eigenvalue for $(i\nabla +{\mathbf A}_{a^-\!,a^+})^2$. 
Then
\begin{equation*}
   \lambda_N^a - \lambda_N
 =\begin{cases}
   \frac{2\pi}{|\log a|}\, |u_N(0)|^2\,(1+o(1)), &\text{if }k=0,\\[4pt]
    C_k\pi\beta^2 a^{2k}\,\sin^2\alpha\,  (1+o(1)), &\text{if }k\ge1,
  \end{cases}
\end{equation*}
as $a\to 0^+$,   $C_k>0$ being a positive constant depending only on $k$.
\end{Theorem}

In the present paper, we are able to remove, in the case $k=0$
(i.e. when the limit eigenfunction $u_N$ does not vanish at the
collision point), the assumption on the symmetry 
of the domain, proving the following result.

\begin{Theorem}{\cite[Theorem 1.17]{AFHL-1}}\label{t:main}
 Let $\Omega$ be 
an open, bounded, and connected set in $\RR^2$ such that $0\in\Omega$. 
Let us assume that there exists $N\geq1$ such that the $N$-th
eigenvalue $\lambda_N$  of the Dirichlet Laplacian in $\Omega$ is 
simple. Let $u_N$ be a $L^2(\Omega)$-normalized eigenfunction 
associated to
$\lambda_N$. If $u_N(0)\neq0$ then 
\begin{equation*}
   \lambda_N^a - \lambda_N
 =\frac{2\pi\, u_N^2(0)}{|\log a|}(1+o(1))
\end{equation*}
as $a\to 0^+$.
\end{Theorem}

It is worthwhile mentioning 
that in \cite{lena} simple magnetic eigenvalues are proved to be
analytic functions of the configuration of the poles provided 
the limit configuration is made of interior distinct poles. 
A consequence of our result is that the latter assumption is even necessary
and simple eigenvalues are not analytic  in a neighborhood of
configurations of poles collapsing outside nodal lines of
the limit eigenfunction.

The proof of Theorem \ref{t:main} relies essentially on 
the characterization of the magnetic eigenvalue as
an eigenvalue of the Dirichlet Laplacian in $\Omega$ with a small set removed,
in the flavor of \cite{AFHL-1} (see \S \ref{sec:reduction}). 
In \cite{AFHL-1} only the case of symmetric domains was considered
and the magnetic problem was shown to be spectrally
equivalent to the eigenvalue problem for the Dirichlet Laplacian  in the domain obtained by removing
 the segment
joining the poles; in the general non-symmetric case, we can still derive a spectral
equivalence with a Dirichlet problem in the domain obtained by
removing from $\Omega$ the nodal lines of magnetic eigenfunctions
close to the collision point. 
The general shape of this removed set (which is not necessarily a
segment as in the symmetric case) creates some
further difficulties; in particular, precise information about the
diameter of such a set is needed in order to apply the following result
from \cite{AFHL-1}.

\begin{Theorem}{\cite[Theorem 1.7]{AFHL-1}}\label{t:one}
  Let $\Omega\subset\RR^2$ be a bounded connected open set containing $0$.
  Let $\lambda_N$ be a simple eigenvalue of the Dirichlet
  Laplacian in $\Omega$ and $u_N$ be a $L^2(\Omega)$-normalized
  eigenfunction  associated to $\lambda_N$   such that $u_N(0)\neq0$.  
Let $(K_{\eps})_{\eps>0}$ be a family of compact connected sets
  contained in $\Omega$ such that, for every $r>0$, there exists
  $\bar\eps$ such that $K_\eps\subseteq D_r$ for every
  $\eps\in(0,\bar\eps)$ ($D_r$ denoting the disk of radius $r$ centered at $0$). 
Then
\[
\lambda_{N}(\Omega\setminus K_\eps)-\lambda_{ N}= u_N^2(0)
\dfrac{2\pi}{|\log
(\mathop{\rm diam}K_\eps)|} +
o\bigg(\dfrac{1}{|\log (\mathop{\rm diam}K_\eps)|}\bigg), \quad \text{as }\eps\to 0,
\]
where $\lambda_{N}(\Omega\setminus K_\eps)$ denotes the $N$-th eigenvalue of the Dirichlet
  Laplacian in $\Omega\setminus K_\eps$. 
\end{Theorem}

In order to apply Theorem \ref{t:one}, a crucial intermediate step in
the proof of Theorem \ref{t:main} is the estimate of the diameter of
nodal lines of magnetic eigenfunctions near the collision point.
More precisely, we prove that, when $a$ is sufficiently small, locally near $0$ 
suitable (magnetic-real) eigenfunctions have 
a nodal set consisting in a single regular curve connecting $a^-$ and $a^+$.
If $d_a$ denotes the diameter of such a curve, we obtain that 
\begin{equation}\label{eq:21}
\lim_{a\to 0^+}\dfrac{|\log a|}{|\log d_a|}=1,
\end{equation}
see \S \ref{sec:proof-theor-reft:m}.

The paper is organized as follows. In section
\ref{sec:estimates-from-above} we obtain some preliminary upper bounds
for the eigenvalue variation $\lambda_N^a-\lambda_N$ testing the
Rayleigh quotient for eigenvalues with proper test functions
constructed  by
suitable manipulation of limit eigenfunctions. In section
\ref{sec:gauge-invar-nodal} we prove that, as the two poles of the
 operator \eqref{eq:12} move towards each other colliding at $0$,
 then  $\lambda_N^a$ is equal to the $N$-th eigenvalue of the
 Laplacian in $\Omega$ with a small piece of nodal line of the
 magnetic eigenfunction removed. Combining the upper estimates of
 section \ref{sec:estimates-from-above} with Theorem \ref{t:one}, in
 section \ref{sec:proof-theor-reft:m} we  succeed in estimating the
 diameter of the removed small set as in 
\eqref{eq:21}; we then conclude the proof of Theorem \ref{t:main} by
combining \eqref{eq:21} and Theorem \ref{t:one}.

\section{Estimates from above}\label{sec:estimates-from-above}
 We denote by $\mathcal H_a$ the closure of
$C^{\infty}_{\rm c}(\Omega\setminus\{a^{+},a^{-}\},\CC)$ with respect to the norm
\[
\|u\|_{\mathcal{H}_a}=\left(\int_{\Omega}\left|(i\nabla+{\mathbf A}_{a^-\!,a^+})u\right|^2\,dx\right)^{1/2}.
\] 
We observe that, by Poincar\'e and diamagnetic
inequalities together with the Hardy type inequality proved in \cite{LW99}, 
$\mathcal H_a\subset H^1_0(\Omega)$ with continuous inclusion. In order
to estimate from above the eigenvalue $\lambda_N^a$, we recall the
well-known Courant-Fisher \emph{minimax characterization}:
\begin{equation}\label{eq:91_la}
  \lambda_N^a =\! \min\bigg\{\!\max_{u\in F\setminus \{0\}}
\dfrac{\int_{\Omega} |(i\nabla+{\mathbf A}_{a^-\!,a^+}) u|^2dx}{\int_{\Omega}
  |u|^2\,dx}:F \text{ is a subspace of $\mathcal H_a$, $\dim F=N$}\bigg\}.
\end{equation}

\begin{Lemma}
\label{lemCutOff}
Let $\tau\in(0,1)$. For every $0<\eps<1$, there exists a continuous
radial
cut-off function
 $\rho_{\eps,\tau}:\RR^2\to \RR$, such that $\rho_{\eps,\tau}\in
 H^1_{\rm loc}(\RR^2)$ and  
\begin{enumerate}[\rm (i)]\itemsep5pt
\item $0\le \rho_{\eps,\tau}(x)\le 1$ for all $x\in\R^2$;
\item $\rho_{\eps,\tau}(x)=0$ if $|x|\le \eps$ and $\rho_{\eps,\tau}(x)=1$ if $|x|\ge \eps^\tau$;
\item
  $\int_{\RR^2}|\nabla\rho_{\eps,\tau}|^2\,dx=\frac{2\pi}{(\tau-1)\log\eps}$;
\item $\int_{\RR^2}(1-\rho_{\eps,\tau}^2)\,dx=O\left(\eps^{2 \tau}\right)$ as
  $\eps\to 0^+$.
\end{enumerate}
      \end{Lemma}

\begin{proof} We set
\begin{equation*}
  \rho_{\eps,\tau}(x)=
\begin{cases}
0, & \mbox{ if } |x| \le \eps,\\
\frac{\log|x|-\log(\eps)}{\log(\eps^\tau)-\log(\eps)},& \mbox{ if }  \eps< |x| <\eps^\tau,\\
1, & \mbox{ if }  |x|\ge \eps^\tau.
\end{cases}
\end{equation*}
The function $\rho_{\eps,\tau}$ is continuous and locally in $H^1$,
with $0 \le \rho_{\eps,\tau} \le 1$. 
The function $1-\rho_{\eps,\tau}^2$ is supported in the disk of
radius $\eps^\tau$ centered at $0$. 
We therefore have 
\[
\int_{\R^2}(1-\rho_{\eps,\tau}^2(x))\,dx \le \pi \eps^{2\tau},
\]
which proves (iv). We have $\nabla \rho_{\eps,\tau}(x)=0$ if $|x|<\eps$ or $|x|>\eps^\tau$, and 
\[
\nabla \rho_{\eps,\tau}(x)=\frac{x}{(\tau-1)\log(\eps)|x|^2}
\]
if $\eps<|x|<\eps^\tau$. From this we directly obtain identity (iii).
\end{proof}

\begin{Lemma}
\label{lemPhase}
 For all $a>0$, there exists a  smooth function $\psi_a: \R^2\setminus s_a\rightarrow \R$ satisfying 
\[
\nabla \psi_a ={\mathbf A}_{a^-\!,a^+},
\]
where $s_a$ is the segment in $\R^2$ defined by $s_a:=\{(t,0)\,:\,
-a\le t\le a\}\,$.
Furthermore, for every $x\in \R^2\setminus\{(0,0)\}$, $\lim_{a\to0^+}\psi_a(x)=0$.
\end{Lemma}

\begin{proof} See \cite[Lemma 3.1]{AFHL-1}.
\end{proof}

The first step in the proof of Theorem \ref{t:main} is the
following upper bound for the eigenvalue~$\lambda_N^a$.

\begin{Proposition}\label{propUpper}
For every $\tau\in(0,1)$
\begin{equation*}
	\lambda_N^a\le \lambda_N+\frac{2\pi}{(1-\tau)|\log a|}
\left(u_N^2(0)+o(1)\right)
\end{equation*}
as $a\to0^+$.
\end{Proposition}

The proof of Proposition \ref{propUpper} is based on estimates from
above of the  Rayleigh quotient for $\lambda_N^a$ computed at some
proper test functions constructed by suitable manipulation of limit
eigenfunctions.  To this aim, let us consider, for each
$j\in \{1,\dots,N\}$, a real eigenfunction $u_j$ of $-\Delta$ with
homogeneous Dirichlet boundary conditions
associated with $\lambda_j$, with
$\|u_j\|_{L^2(\Omega)}=1$. Furthermore, we choose these eigenfunctions so that
 \begin{equation}\label{eq:7}
\int_{\Omega}u_j{u_k}\,dx =0 \mbox{ for } j\neq k.
\end{equation}
 For $j\in \{1,\dots,N\}$ and $a>0$ small enough, we set 
\begin{equation}\label{eq:10}
v_{j,\tau}^a:=e^{i\psi_a}\rho_{2a,\tau}u_j.
\end{equation}
We have that $v_{j,\tau}^a \in \mathcal{H}_a$. 
Lemma \ref{lemCutOff} and the 
Dominated Convergence Theorem 
imply that $v_{j,\tau}^a $ tends to $u_j$ in $L^2(\Omega)$ when
$a\to 0\,^+$. This implies in particular that the functions $v_{j,\tau}^a $
are linearly independent for $a$ small enough.
 
Hence, for $a>0$ small enough, 
$E_{N,\tau}^a=\mathop{\rm
  span}\big\{v_{1,\tau}^a,\dots,v_{N,\tau}^a\big\}$ is an
$N$-dimensional subspace of $\mathcal H_a$, so that, in view of
\eqref{eq:91_la},
\begin{equation}\label{eq:2}
\lambda_N^a\leq \max_{u\in E_{N,\tau}^a\setminus\{0\}}\frac{\int_{\Omega} |(i\nabla+{\mathbf A}_{a^-\!,a^+}) u|^2dx}{\int_{\Omega}
  |u|^2\,dx}
=\frac{\int_{\Omega} |(i\nabla+{\mathbf A}_{a^-\!,a^+}) v_\tau^a|^2dx}{\int_{\Omega}
  |v_\tau^a|^2\,dx}
\end{equation}
with 
\begin{equation}\label{eq:1}
v_\tau^a=\sum_{j=1}^N \alpha_{j,\tau}^a v_{j,\tau}^a\quad\text{for
some }\alpha_{1,\tau}^a,\dots \alpha_{N,\tau}^a\in\C \text{ such
that }\sum_{j=1}^N|\alpha_{j,\tau}^a|^2=1.
\end{equation}
 
\begin{Lemma}\label{l:stimadenom} 
 For $a>0$ small, let $v_\tau^a$ be as in \eqref{eq:2}--\eqref{eq:1}
 above. Then
 \begin{equation}\label{eq:9}
   \int_{\Omega} |v_\tau^a|^2dx=1+O(a^{2\tau})
 \end{equation}
as $a\to0^+$.
 \end{Lemma}
  \begin{proof}
Taking into account \eqref{eq:1}, \eqref{eq:10}, and \eqref{eq:7}, we
can write
\begin{align*}
  \int_{\Omega}
  |v_\tau^a|^2dx&=\sum_{j,k=1}^N\alpha_{j,\tau}^a\overline{\alpha_{k,\tau}^a}
                    \int_\Omega \rho_{2a,\tau}^2u_ju_k\,dx\\
&=1+\sum_{j=1}^N|\alpha_{j,\tau}^a|^2 \int_\Omega (\rho_{2a,\tau}^2-1)u_j^2\,dx
+\sum_{j\neq k}\alpha_{j,\tau}^a\overline{\alpha_{k,\tau}^a}
                    \int_\Omega (\rho_{2a,\tau}^2-1)u_ju_k\,dx.
\end{align*}
Hence the conclusion follows from Lemma \ref{lemCutOff} (iv). 
  \end{proof}
 \begin{Lemma}\label{lemIdentity} 
 For $a>0$ small, let $v_\tau^a$ be as in \eqref{eq:2}--\eqref{eq:1}
 above. Then
 \begin{multline}\label{eqIdentityLambdak}	
   \int_{\Omega} |(i\nabla+{\mathbf A}_{a^-\!,a^+}) v_\tau^a|^2dx\\
   =\sum_{j,k=1}^N \alpha_{j,\tau}^a\overline{\alpha_{k,\tau}^a}
   \left(\frac{\lambda_j+\lambda_k}{2}\int_{\Omega\setminus D_{2a}}
     \rho_{2a,\tau}^2u_ju_k\,dx+\int_{D_{(2a)^\tau}\setminus
         D_{2a} }u_ju_k\left|\nabla\rho_{2a,\tau}\right|^2\,dx\right),
   \end{multline}
 where, for all $r>0$, $D_r=\{(x_1,x_2)\in\R^2:x_1^2+x_2^2<r\}$
 denotes the disk of center $(0,0)$ and radius $r$.
 \end{Lemma}
  \begin{proof} 
    Let us fix $j$ and $k$ in $\{1,\dots,N\}$ (possibly equal). We
    have that, in $\Omega\setminus D_{2a}$,
  \begin{align*}
(i\nabla+{\mathbf A}_{a^-\!,a^+})&v_{j,\tau}^a\cdot\overline{ (i\nabla+{\mathbf A}_{a^-\!,a^+})v_{k,\tau}^a}\\
&=\nabla (\rho_{2a,\tau}u_j)\cdot \nabla (\rho_{2a,\tau}u_k)\\
&=
 \rho_{2a,\tau}^2\nabla u_j\cdot \nabla u_k +u_j u_k \left|\nabla
  \rho_{2a,\tau}\right|^2
+\left(u_j\nabla u_k+u_k\nabla u_j\right)\cdot \rho_{2a,\tau}\nabla \rho_{2a,\tau},       
  \end{align*}
  and, since $\rho_{2a,\tau}\nabla \rho_{2a,\tau}=\frac12\nabla(\rho_{2a,\tau}^2)$,
  \begin{multline}  \label{eqDerivij}
  \int_{\Omega}
(i\nabla+{\mathbf A}_{a^-\!,a^+})v_{j,\tau}^a\cdot\overline{
  (i\nabla+{\mathbf A}_{a^-\!,a^+})v_{k,\tau}^a}\,dx= 
\int_{\Omega\setminus D_{2a}}\rho_{2a,\tau}^2\nabla u_j\cdot \nabla
u_k\,dx \\+
 \int_{D_{(2a)^\tau}\setminus
         D_{2a} }
u_j u_k \left|\nabla
  \rho_{2a,\tau}\right|^2\,dx+\frac12\int_{\Omega\setminus
         D_{2a} }\left(u_j\nabla u_k+u_k\nabla u_j\right)\cdot \nabla(\rho_{2a,\tau}^2)\,dx.       
  \end{multline}
 An integration by part on the last term of \eqref{eqDerivij} gives us
  \begin{multline*}
    \int_{\Omega}
    (i\nabla+{\mathbf A}_{a^-\!,a^+})v_{j,\tau}^a\cdot\overline{
    	(i\nabla+{\mathbf A}_{a^-\!,a^+})v_{k,\tau}^a}\,dx= 
    \int_{\Omega\setminus D_{2a}}\rho_{2a,\tau}^2\nabla u_j\cdot \nabla
    u_k\,dx \\+
    \int_{D_{(2a)^\tau}\setminus
    	D_{2a} }
    u_j u_k \left|\nabla
    \rho_{2a,\tau}\right|^2\,dx-\frac12\int_{\Omega\setminus
    	D_{2a} }\left(u_j\Delta u_k+2\nabla u_k\cdot\nabla u_j+\Delta u_j\, u_k\right)\rho_{2a,\tau}^2\,dx.
  \end{multline*}
After cancellations, we get
  \begin{multline}
  \label{eqScalarProdij}
  \int_{\Omega}
(i\nabla+{\mathbf A}_{a^-\!,a^+})v_{j,\tau}^a\cdot\overline{
  (i\nabla+{\mathbf A}_{a^-\!,a^+})v_{k,\tau}^a}\,dx\\= 
\frac{\lambda_k+\lambda_j}{2} \int_{\Omega\setminus
      D_{2a}}\rho_{2a,\tau}^2u_ju_k\,dx+
 \int_{D_{(2a)^\tau}\setminus
         D_{2a} }
u_j u_k \left|\nabla
  \rho_{2a,\tau}\right|^2\,dx.
\end{multline}
From \eqref{eq:1}, bilinearity, and \eqref{eqScalarProdij} we obtain
\eqref{eqIdentityLambdak}.
  \end{proof}
From \eqref{eq:2} and \eqref{eqIdentityLambdak} it follows that 
\begin{equation}\label{eq:4}
  \lambda_N^a-\lambda_N\leq\frac1 {\int_{\Omega}
  |v_\tau^a|^2\,dx}\left[\mathcal Q_a( \alpha_{1,\tau}^a,
  \alpha_{2,\tau}^a,\dots, \alpha_{N,\tau}^a)
+\lambda_N\left(1-\int_{\Omega}
  |v_\tau^a|^2\,dx \right)\right]
\end{equation}
where $\mathcal Q_a:\C^N\to\R$ is the quadratic form defined as
\begin{equation}\label{eq:QF1}
\mathcal Q_a(z_1,z_2,\dots,z_N)=\sum_{j,k=1}^N
M_{jk}^a
z_j\overline{z_k}
\end{equation}
with
\begin{equation}\label{eq:QF2}
M_{jk}^a
=
\frac{\lambda_j+\lambda_k}{2}\int_{\Omega\setminus D_{2a}}
     \rho_{2a,\tau}^2u_ju_k\,dx+\int_{D_{(2a)^\tau}\setminus
         D_{2a}
       }u_ju_k\left|\nabla\rho_{2a,\tau}\right|^2\,dx-\lambda_N \delta_{jk}
     \end{equation}
being $\delta_{jk}$ the Kronecker delta.

To estimate the largest eigenvalue of the quadratic form $\mathcal
Q_a$, we will use the following technical lemma.

\begin{Lemma}\label{l:tech}
  For every  $\eps>0$ let us consider a quadratic form 
\[
Q_\eps:\C^{N}\to \R,\quad 
Q_\eps(z_1,z_2,\dots,z_{N})=\sum_{j,k=1}^{N}m_{j,k}(\eps)z_j
\overline{z_k},
\]
with $m_{j,k}(\eps)\in\C$ such that
$m_{j,k}(\eps) =\overline{m_{k,j}(\eps)}$. Let us assume that there
exist real numbers $C>0$ and $K_1,K_2,\dots,K_{N-1}<0$ such that 
\begin{align*}
&m_{N,N}(\eps)=C\eps(1+o(1)) \text{ as }\eps\to 0^+,\\
& m_{j,j}(\eps)=K_j+o(1)\text{ as }\eps\to 0^+\text{ for all }j<N,\\
& 
m_{j,k}(\eps)=\overline{m_{k,j}(\eps)}=O(\eps)\text{ as }\eps\to
               0^+\text{ for all }j\neq k .
\end{align*}
Then 
\[
\max\bigg\{Q_\eps(z_1,\dots,z_N):(z_1,\dots,z_N)\in\C^N,\sum_{j=1}^N|z_j|^2=1\bigg\}
=C\eps(1+o(1))\quad\text{as }\eps\to0^+.
\]
\end{Lemma}
\begin{proof}
  The result is contained in \cite[Lemma 6.1]{AF}, hence we omit the proof. 
\end{proof}

\begin{Lemma}\label{l:maxQF}
  For $a>0$ small, let $\mathcal Q_a:\C^N\to\R$ be the quadratic form
  defined in \eqref{eq:QF1}--\eqref{eq:QF2}. Then
\[
\max\bigg\{\mathcal
  Q_a(z_1,\dots,z_N):(z_1,\dots,z_N)\in\C^N,\sum_{j=1}^N|z_j|^2=1\bigg\}=
 \frac{2\pi u_N^2(0)}{(1-\tau)|\log(a)|} \left
     (1+o(1)\right)
\]
as $a\to 0^+$.
\end{Lemma}
\begin{proof}
  Since $\int_\Omega u_N^2=1$, we can write
\[
   M_{NN}^a= 
\lambda_N\int_{\Omega}
     (\rho_{2a,\tau}^2-1)u_N^2\,dx+\int_{D_{(2a)^\tau}\setminus
         D_{2a}
       }u_N^2|\nabla\rho_{2a,\tau}|^2\,dx.
\]
Since $u_N\in L^\infty_{\rm loc}(\Omega)$, from Lemma \ref{lemCutOff} (iv) it follows that 
\[
\int_{\Omega}
     (\rho_{2a,\tau}^2-1)u_N^2\,dx=\int_{D_{(2a)^\tau}}
     (\rho_{2a,\tau}^2-1)u_N^2\,dx
=O(a^{2\tau})\quad\text{as
     }a\to0^+.
\]
Since $u_N\in C^\infty_{\rm loc}(\Omega)$ we have that
$u_N^2(x)-u_N^2(0)=O(|x|)$ as $|x|\to0^+$, then Lemma \ref{lemCutOff}
(iii) implies that
\begin{align*}
\int_{D_{(2a)^\tau}\setminus
         D_{2a}
       }&u_N^2|\nabla\rho_{2a,\tau}|^2\,dx\\
&=
u_N^2(0)\int_{D_{(2a)^\tau}\setminus
         D_{2a}
       }|\nabla\rho_{2a,\tau}|^2\,dx+\int_{D_{(2a)^\tau}\setminus
         D_{2a}
       }(u_N^2(x)-u_N^2(0))|\nabla\rho_{2a,\tau}(x)|^2\,dx\\
&=(u_N^2(0)+O(a^\tau)) \int_{D_{(2a)^\tau}\setminus
         D_{2a}
  }|\nabla\rho_{2a,\tau}|^2\,dx\\
&=\frac{2\pi}{(\tau-1)\log(2a)}(u_N^2(0)+O(a^\tau))
=\frac{2\pi}{(\tau-1)\log(a)}u_N^2(0) \left (1+o(1)\right)
\end{align*}
as $a\to0^+$.
Then 
\begin{equation}\label{eq:5}
   M_{NN}^a= \frac{2\pi}{(\tau-1)\log(a)}u_N^2(0) \left
     (1+o(1)\right)
\quad\text{as }a\to 0^+.
\end{equation}
For all $1\leq j <N$ we have that 
\begin{align*}
 M_{jj}^a&= 
\lambda_j\int_{\Omega\setminus D_{2a}}
     \rho_{2a,\tau}^2u_j^2\,dx+\int_{D_{(2a)^\tau}\setminus
         D_{2a}
       }u_j^2\left|\nabla\rho_{2a,\tau}\right|^2\,dx-\lambda_N\\
&=(\lambda_j-\lambda_N)+\lambda_j\int_{\Omega}
     (\rho_{2a,\tau}^2-1)u_j^2\,dx+\int_{\Omega}u_j^2\left|\nabla\rho_{2a,\tau}\right|^2\,dx
\end{align*}
and hence, since $u_j\in C^\infty_{\rm loc}(\Omega)$ and in view of
Lemma \ref{lemCutOff},
\begin{equation}\label{eq:6}
  M_{jj}^a= (\lambda_j-\lambda_N)+O\left(\frac1{|\log a|}\right)
= (\lambda_j-\lambda_N)+o(1)
 \quad\text{as }a\to 0^+.
\end{equation}
Moreover, for all $j,k=1,\dots,N$ with $j\neq k$, in view of
\eqref{eq:7} and Lemma \ref{lemCutOff} 
we have that 
\begin{equation}\label{eq:8}
  M_{jk}^a
  =
  \frac{\lambda_j+\lambda_k}{2}\int_{\Omega\setminus D_{2a}}
  (\rho_{2a,\tau}^2-1)u_ju_k\,dx+\int_{D_{(2a)^\tau}\setminus
    D_{2a}
  }u_ju_k\left|\nabla\rho_{2a,\tau}\right|^2\,dx=O\left(\frac1{|\log a|}\right) 
\end{equation}
as $a\to0^+$.

 In view of estimates \eqref{eq:5}, \eqref{eq:6}, and
\eqref{eq:8}, we have that $\mathcal Q_a$ satisfies the assumption of
Lemma \ref{l:tech} (with $\eps=\frac1{|\log a|}$), hence the
conclusion follows from Lemma \ref{l:tech}.
\end{proof}

\begin{proof}[Proof of Proposition  \ref{propUpper}]
Combining \eqref{eq:4}, Lemma \ref{l:maxQF}, and estimate \eqref{eq:9}
we obtain that
\begin{align*}
  \lambda_N^a-\lambda_N&\leq\frac1{1+O(a^{2\tau})}\left[
 \frac{2\pi u_N^2(0)}{(1-\tau)|\log(a)|} \left
     (1+o(1)\right)+O(a^{2\tau})\right]\\
&=\frac{2\pi u_N^2(0)}{(1-\tau)|\log(a)|} \left
     (1+o(1)\right)\quad\text{as }a\to0^+
\end{align*}
thus completing the proof.  
\end{proof}

\section{Gauge invariance, nodal sets and reduction to the
  Dirichlet-Laplacian}\label{sec:gauge-invar-nodal}

In the following, we mean by a \emph{path} $\gamma$ a piecewise-$C^1$
map $\gamma: I \mapsto \R^2$, with $I=[a,b]\subset \R$ a closed
interval. It follows from the definition of ${\mathbf A}_{a^-\!,a^+}$ (see \eqref{eq:11})
that for any closed path $\gamma$ (i.e. $\gamma(a)=\gamma(b)$)
\begin{equation}
 \label{eqWinding}
 \frac{1}{2\pi}\oint_{\gamma} {\mathbf A}_{a^-\!,a^+}\cdot d\mathbf{s}=
 \frac{1}{2}\mathop{\rm ind}\nolimits_{\gamma}(a^+)-\frac{1}{2}\mathop{\rm ind}\nolimits_{\gamma}(a^-),
\end{equation}
where $\mathop{\rm ind}\nolimits_{\gamma}(a^+)$ (resp. $\mathop{\rm
  ind}\nolimits_{\gamma}(a^-$))  is the winding number of $\gamma$ around $a^+$ (resp. $a^-$).

\subsection{Gauge invariance}\label{sec:gauge-invariance}
Let us give some results concerning the gauge invariance of our
operators. In view of applying them to several different situations,
we give statements valid for a magnetic Hamiltonian in an open and
connected domain $D$, without restricting ourselves to the
Aharonov-Bohm case.

 In the following, the term \emph{vector potential} (in  an open
 connected domain $D$)
 stands for a smooth real vector field ${\mathbf A}:D\to\R^2$. In
 order to define the quantum mechanical Hamiltonian for a particle in
 $D$, under the action of the magnetic field derived from the vector
 potential $\mathbf{A}$, we first consider the differential operator
\begin{equation*}
P=\left(i\nabla+\mathbf{A}\right)^2,
\end{equation*}
acting on smooth functions compactly supported in $D$. Using
integration by parts (Green's formula), one can easily see that $P$ is
symmetric and positive. This is formally the desired Hamiltonian, but
to obtain a self-adjoint Schr\"odinger operator, we have to specify
boundary conditions on $\partial D$, which we choose to be Dirichlet
boundary conditions everywhere. More specifically, our Hamiltonian is
the Friedrichs extension of the differential operator $P$. We denote
it by $H_{\mathbf{A}}^D$, and we call it the magnetic Hamiltonian on
$D$ associated with $\mathbf{A}$.

We observe that the Aharonov-Bohm operator $H_{a^-\!,a^+}^{\Omega}$
with poles $a^-=(-a,0)$, $a^+=(a,0)$ in $\Omega$ introduced in
\eqref{eq:12} can be defined as the magnetic Hamiltonian
$H_{\mathbf{A}_{a^-\!,a^+}}^{\dot\Omega}$ on $\dot\Omega$, where
$\dot\Omega=\Omega\setminus\{a^-,a^+\}$, and that the
  spectrum of $H_{a^-\!,a^+}^{\Omega}$ consists of the eigenvalues
  defined by \eqref{eq:91_la}. The space $\mathcal{H}_a$ is
  the form domain of $H^{\Omega}_{a_-,a^+}$.

\begin{Definition}\label{d:gauge-invariance}
  We call \emph{gauge function} a smooth complex valued function
  $\psi:D\to\C$ such that $|\psi|\equiv 1$.  To any gauge
  function $\psi$, we associate a \emph{gauge transformation} acting
  on pairs magnetic potential-function as
  $(\mathbf{A},u)\mapsto(\mathbf{A}^*,u^*)$, with
\begin{equation*}
\begin{cases}
&\mathbf{A}^*=\mathbf{A}-i\frac{\nabla\psi}{\psi},\\
&u^*= \psi u
\end{cases}
\end{equation*}
where $\nabla \psi=\nabla(\Re\psi)+i \nabla(\Im\psi)$. We notice that,
since $|\psi|=1$, $i\frac{\nabla\psi}{\psi}$ is a real vector field.
Two magnetic potentials are said to be gauge equivalent if one can be obtained from the other by a gauge transformation (this is an equivalence relation).
      \end{Definition}

The following result is a consequence \cite[Theorem 1.2]{Lei83}.

\begin{Proposition}\label{p:unieq}
	If $\mathbf{A}$ and $\mathbf{A}^*$ are two gauge equivalent vector potentials, the operators $H_{\mathbf{A}}^D$ and $H_{\mathbf{A}^*}^D$ are unitarily equivalent.
\end{Proposition}

The equivalence between two vector potentials (which is equivalent
to the fact that their difference 
is gauge-equivalent to $0$) can be determined using the following criterion.
\begin{Lemma}\label{lemEquivPot}
Let $\mathbf{A}$ be a vector potential in $D$. It is gauge equivalent to $0$ if and only if
\begin{equation}\label{eqIntCirc}
  \frac{1}{2\pi}\oint_{\gamma} \mathbf{A}(s)\cdot d\mathbf{s} \in\Z
\end{equation}  
for every closed path $\gamma$ contained in $D$.
\end{Lemma}

  \begin{remark} The reverse implication in Lemma
    \ref{lemEquivPot} is contained in
    \cite[Theorem 1.1]{HHOO99}, for the Neumann boundary condition.
\end{remark}

\begin{proof} Let us first prove the direct implication. We assume that $\mathbf{A}$ is gauge equivalent to $0$, that is to say that there exists a gauge function $\psi$ such that 
\begin{equation*}
  \mathbf{A}\equiv i\frac{\nabla \psi}{\psi}.
\end{equation*}
Let us fix a closed path $\gamma: I=[a,b] \to D$ and consider the
mapping $z=\psi\circ \gamma$ from $I$ to $\mathbb{U}$, where
$\mathbb{U}=\{z \in \C\,:\, |z|=1\}$.  By the lifting property, there
exists a piecewise-$C^1$ function $\theta: I \to \R$ such that
$z(t)=\exp(i\theta(t))$ for all $t\in I$. This implies that
\begin{equation*}
	\nabla\psi(\gamma(t)) \cdot\gamma'(t) = \left(\psi\circ\gamma\right)'(t)=z'(t)=i\theta'(t)\exp(i\theta(t)),
\end{equation*} 
and therefore
\begin{equation*}
	i\frac{\nabla \psi}{\psi}\left(\gamma(t)\right) \cdot\gamma'(t)=-\theta'(t).
\end{equation*}
This implies that 
\begin{equation*}
	\oint_{\gamma} \mathbf{A}(s)\cdot d\mathbf{s}=\int_a^b i\frac{\nabla \psi}{\psi}\left(\gamma(t)\right)\cdot \gamma'(t)\,dt=-\int_a^b \theta'(t)\,dt=\theta(a)-\theta(b).
\end{equation*}
Since $\gamma$ is a closed path, $\exp(i \theta(a))=\exp(i \theta(b))$, and therefore
\begin{equation*}
\frac{\theta(a)-\theta(b)}{2\pi} \in\Z.
\end{equation*}  

Let us now consider the reverse implication. We define a gauge
function $\psi$ in the following way. We fix an (arbitrary) point
$X_0=(x_0,y_0)\in D$. Let us show that, for $X=(x,y)\in D$, the
quantity
\begin{equation*}
	\exp\left(-i\int_{\gamma}\mathbf{A}(s)\,d\mathbf{s}\right)
\end{equation*}
does not depend on the choice of a path $\gamma$ from $X_0$ to
$X$. Indeed, let $\gamma_1$ and $\gamma_2$ be two such paths, and let
$\gamma_3$ be the closed path obtained by going from $X_0$ to $X$
along $\gamma_1$ and then from $X$ to $X_0$ along $\gamma_2$. On the
one hand, we have
\begin{equation*}
  \oint_{\gamma_3}\mathbf{A}(s)\,d\mathbf{s}=
\int_{\gamma_1}\mathbf{A}(s)\,d\mathbf{s}-\int_{\gamma_2}\mathbf{A}(s)\,d\mathbf{s}.
\end{equation*} 
On the other hand, if \eqref{eqIntCirc} holds, we have
	\begin{equation*}
	\oint_{\gamma_3} \mathbf{A}(s)\,d\mathbf{s} \in2\pi\Z.
	\end{equation*}
This implies that 
\begin{equation*}
	\exp\left(-i\int_{\gamma_1}\mathbf{A}(s)\,d\mathbf{s}\right)=
	\exp\left(-i\int_{\gamma_2}\mathbf{A}(s)\,d\mathbf{s}\right).
	\end{equation*}
        By connectedness of $D$, there exists a path from $X_0$ to $X$
        for any $X\in \Omega$ (we can even choose it piecewise
        linear). We can therefore define, without ambiguity, a
        function $\psi: \Omega\to \C$ by
\begin{equation*}
\psi(X)= \exp\left(-i\int_{\gamma}\mathbf{A}(s)\,d\mathbf{s}\right).
\end{equation*}
It is immediate from the definition that $|\psi|\equiv 1$ and that $\psi$ is smooth, with
\begin{equation*}
	\nabla\psi(X)=-i\psi(X)\mathbf{A}(X).
\end{equation*} 	      
It is therefore a gauge function sending $\mathbf{A}$ to $0$. \end{proof}

Lemma \ref{lemEquivPot} can be used to define a set of eigenfunctions
for $H_{a^-\!,a^+}^{\Omega}$ having especially nice properties. It is
analogous to the set of real eigenfunctions for the usual
Dirichlet-Laplacian. To define it, we will construct a conjugation,
that is an antilinear antiunitary operator, which commutes with
$H_{a^-\!,a^+}^{\Omega}$. To simplify notation, we denote
$\mathbf{A}_{a^-\!,a^+}$ by $\mathbf{A}$ and $H^{\Omega}_{a^-\!,a^+}$ by
$H$ in the rest of this section.

According to \eqref{eqWinding}, the vector potential $2\mathbf{A}$
satisfies  condition \eqref{eqIntCirc} of Lemma \ref{lemEquivPot} on $\dot\Omega$,
and therefore is gauge equivalent to $0$. Therefore there  exists a
gauge function $\psi$ in $\dot\Omega$ such that
\begin{equation*}
	 2\mathbf{A}= -i\frac{\nabla\psi}{\psi}\quad\text{in }\dot\Omega.
\end{equation*}  
 We now define the antilinear antiunitary operator $K$ by
\begin{equation*}
	Ku= \psi \bar{u}.
\end{equation*}
For all $u \in C^{\infty}_0(\dot{\Omega},\C)$,
\begin{equation*}
  \left(i\nabla+ \mathbf{A}\right)(\psi \bar{u})=
  \psi\left(i\nabla+i\frac{\nabla\psi}{\psi}+\mathbf{A}\right)\bar{u}
  =\psi\left(i\nabla-\mathbf{A}\right)\bar{u}=-\psi\overline{\left(i\nabla+\mathbf{A}\right)u}.
\end{equation*}
The above formula, and the fact that $K$ is antilinear and antiunitary,
imply that, for all $u$ and $v$ in $C^{\infty}_0(\dot{\Omega},\C)$,
\begin{align*}
\langle K^{-1}HKu,v \rangle= \langle Kv, HKu\rangle
&=\int_{\Omega}\left(i\nabla +\mathbf{A}\right)(\psi \bar{v}) 
\cdot \overline{\left(i\nabla +\mathbf{A}\right)(\psi \bar{u})}\,dx\\
&=\int_{\Omega}\overline{\left(i\nabla +\mathbf{A}\right)v} \cdot \left(i\nabla +\mathbf{A}\right)u\,dx
= \langle Hu,v\rangle,
\end{align*}
where $\langle f,g\rangle=\int_\Omega f\bar g\,dx$ denotes the standard  scalar
product on the 
complex Hilbert space $L^2(\Omega,\C)$.
By density, we conclude that  
\[
K^{-1}HK=H.
\]

\begin{Definition}
  We say that a function $u\in L^2(\Omega,\C)$ is \emph{magnetic-real}
  when $Ku=u$.
\end{Definition}

Let denote by $\mathcal{R}$ the set of magnetic-real functions in
$L^2(\Omega,\C)$. The restriction of the scalar product to
$\mathcal{R}$ gives it the structure of a real Hilbert space. The
commutation relation $HK=KH$ implies that $\mathcal{R}$ is stable
under the action of $H$; we denote by $H^R$ the restriction of $H$ to
$\mathcal{R}$. There exists an orthonormal basis of $\mathcal{R}$
formed by eigenfunctions of $H^R$. Such a basis can be seen as a basis
of magnetic-real eigenfunctions of the operator $H$, in the complex
Hilbert space $L^2(\Omega,\C)$.
 
Let us now fix an eigenfunction $u$ of $H^R$ (or, equivalently, a
magnetic-real eigenfunction of $H$).  We define its \emph{nodal set}
$\mathcal{N}(u)$ as the closure in $\overline \Omega$ of the zero-set
$u^{-1}(\{0\})$. Let us describe the local structure of
$\mathcal{N}(u)$. In the sequel, by a \emph{regular curve} or 
\emph{regular arc} we mean a curve admitting a $C^{1,\alpha}$
parametrization, for some $\alpha\in(0,1)$.

\begin{Theorem}\label{thmNodalLoc}
 	The set $\mathcal{N}(u)$ has the following properties.
\begin{enumerate}[\rm (i)]
\item $\mathcal{N}(u)$ is, locally in $\dot{\Omega}$, a regular curve, except
  possibly at a finite number of singular points
  $\{X_j\}_{j\in\{1,\dots, n\}}$.
\item For $j\in \{1,\dots,n\}$, in the neighborhood of $X_j$,
  $\mathcal{N}(u)$ consists in an \emph{even} number of regular
  half-curves meeting at $X_j$ with equal angles (so that $X_j$ can be
  seen as a cross-point).
\item In the neighborhood of $a^+$ (resp. $a^-$), $\mathcal{N}(u)$
  consists in an \emph{odd} number of regular half-curves meeting at
  $a^+$ (resp. $a^-$) with equal angles (in particular this means that
  $a^+$ and $a^-$ are always contained in $\mathcal{N}(u)$).
\end{enumerate}		
\end{Theorem}

\begin{proof} 
The proof is essentially contained in \cite[Theorem 1.5]{NT}; 
for the sake of completeness we present a sketch of it.
Let the eigenfunction $u$ be associated with the
  eigenvalue $\lambda$, so that $Hu=\lambda u$. Let $x_0$ be a point
  in $\dot{\Omega}$. For $\eps>0$, we denote by $D(x_0,\eps)$ the open
  disk $\{x\,:\,|x-x_0|<\eps\}$. Let us show that we can find $\eps>0$
  small enough and a local gauge transformation
  $\varphi: D(x_0,\eps)\to \CC$ such that
  $\mathbf A^*=\mathbf A-i\frac{\nabla \varphi}{\varphi}=0$ and
  $u^*=\varphi u$ is a real-valued function in $D(x_0,\eps)$. Indeed,
  let us define, as before, the gauge function $\psi$ such that
  $2\mathbf A=-i\frac{\nabla \psi}{\psi}$. For $\eps>0$ small enough,
  we can define a smooth function $\varphi: D(x_0,\eps)\to \CC$ such
  that $\overline{\psi}(x)=(\varphi(x))^2$ for all $x\in D(x_0,\eps)$:
  take
\[
\varphi(x)=\exp\left(-\frac i2\mbox{arg}(\psi(x))\right)
\] 
with $\mbox{arg}$ a determination of the argument in
$\psi(D(x_0,\eps))$. A direct computation shows that, for
$x \in D(x_0,\eps)$,
\[
i\frac{\nabla
  \varphi(x)}{\varphi(x)}=\frac{i}2\frac{\nabla\overline{\psi}(x)}{\overline{\psi}(x)}
=\mathbf A(x).
\] 
The gauge transformation on $D(x_0,\eps)$ associated with $\varphi$
therefore sends $\mathbf A$ to $0$. Furthermore, since $u$ is
$K$-real, we have $\psi \bar u=\overline{\varphi^2 u}=u$ in
$D(x_0,\eps)$, and therefore $\overline{\varphi u}=\varphi u$. The
real-valued function $v=\varphi u$ satisfies $-\Delta v=\lambda v$,
and, since $|\varphi|\equiv 1$ on $D(x_0,\eps)$,
we have that $\mathcal N(v)\cap D(x_0,\eps)=\mathcal N(u)\cap D(x_0,\eps)$. Points
(i) and (ii) of Theorem \ref{thmNodalLoc} then follow from classical
results on the nodal set of Laplacian eigenfunctions (see for instance
\cite[Theorem 2.1]{HHOT09} and \cite[Theorem 4.2]{NT}).

To prove point (iii) of Theorem \ref{thmNodalLoc}, we use the
regularity result of \cite{NT} for the Dirichlet problem associated
with a one-pole Aharonov-Bohm operator. Indeed, let $\eps>0$ be small
enough so that $D=D(a^+,\eps)\subset \Omega$ and $a^-\notin D$. By
this choice of $\eps$, $\mathbf A_{a^-}=\nabla f$ on $D$, with $f$ a
smooth function, so that the domain $D$ and the magnetic potential
$\mathbf A$, restricted to $D$, satisfy the hypotheses of
\cite[Theorem 1.5]{NT}. The function $u$ is a solution of the Dirichlet problem
\begin{equation*}
\begin{cases}
  (i\nabla+\mathbf A)^2u-\lambda u=0,&\mbox{in }D,\\
  u=\gamma,&\mbox{on } \partial D,
\end{cases}
\end{equation*} 
with $\gamma=u_{|\partial D} \in W^{1,\infty}(\partial D)$. A direct application of \cite[Theorem 1.5]{NT} gives property (iii) around $a^+$. We obtain property (iii) around $a^-$ by exchanging the role of $a^+$ and $a^-$. \end{proof}

 \subsection{Reduction to the Dirichlet-Laplacian}\label{sec:reduction}
 Our aim in this subsection is to show that, as the two poles of the
 operator \eqref{eq:12} coalesce into a point at which $u_N$ does not
 vanish, then  $\lambda_N^a$ is equal to the $N$-th eigenvalue of the
 Laplacian in $\Omega$ with a small subset concentrating at $0$ removed.

 \begin{Theorem}\label{t:reduct-dirichl-lapl}
Let us assume that there exists $N\geq1$ such that the $N$-th
eigenvalue $\lambda_N$  of the Dirichlet Laplacian in $\Omega$ is 
simple. Let $u_N$ be a $L^2(\Omega)$-normalized eigenfunction associated
$\lambda_N$ and assume that $u_N(0)\neq0$.
Then, for all $a>0$ sufficiently small,  there exists
a compact connected set $K_a \subset \Omega$ such that 
 \begin{equation*}
\lambda_N^a=\lambda_N(\Omega\setminus K_a)
 \end{equation*}
and $K_a$ concentrates around $0$ as
 $a\to0^+$, i.e. for any $\eps>0$ there exists $\delta>0$ such that if
 $a< \delta$ then $K_a\subset D_\eps$.
   
 \end{Theorem}

 We will divide the proof into two lemmas.
 
 \begin{Lemma}\label{lemRing}
   Let $R>0$ be such that $\overline{D_R}\subset \Omega$ and
   $u_N(x)\neq 0$ for all $x\in \overline{D_R}$. Let $r \in (0,R)$.
   We denote by $C_{r,R}$ the closed ring
   \[
   C_{r,R}=\{x\in \R^2: r\le|x|\le R\}.
   \]
   There exists $\delta>0$ such that, if $0<a<\delta$ and
   if $u$ is a magnetic-real eigenfunction associated with
   $\lambda_N^a$, then $u$ does not vanish in
   $C_{r,R}$.
\end{Lemma}
 
 \begin{proof}
   Let us assume, by contradiction, that there exists a sequence
   $a_n\to0^+$ such that, for all $n\ge 1$,
   $\lambda_N^{a_n}$ admits an eigenfunction $\varphi_n$ which
   vanishes somewhere in $C_{r,R}$. Let us denotes by $X_n$ a zero of $\varphi_n$
   in $C_{r,R}$.
 	
   According to \cite[Section III]{lena}, we can assume, up to
   extraction and a suitable normalization of $\varphi_n$, that
   $\varphi_n \to u_N$ in $L^2(\Omega)$.
 Since $H$ is a uniformly
   regular elliptic operator in a neighborhood of $C_{r,R}$, $\varphi_n$
   converges to $u_N$ uniformly on $C_{r,R}$.
Furthermore, up to
   one additional extraction, we can assume that
   $X_n\to X_{\infty}\in C_{r,R}$. This implies that
   $u_N(X_{\infty})=0$, contradicting the fact that $u_N(x)\neq 0$ for all $x\in \overline{D_R}$.
 \end{proof}
 
\begin{Lemma}\label{l:single_curve}
   For all $R>0$
  such that $\overline{D_R}\subset \Omega$
   and $u_N(x)\neq 0$ for all $x\in \overline{D_R}$, there exists
   $\delta>0$ such that, if $0<a<\delta$ and if $u_N^a$ is a
   magnetic-real eigenfunction associated with
   $\lambda_N^a$, then $\mathcal{N}(u_N^a)\cap D_R$ consists
   in a single regular curve connecting $a^-$ and $a^+$.
\end{Lemma}

\begin{proof} By continuity of
    $(a^-,a^+)\mapsto \lambda_N^a$ (see \cite{lena}), we have
    that 
\begin{equation}\label{eq:15}
  \Lambda=\max_{a\in[0,R]}\lambda_N^a\in(0,+\infty). 
 \end{equation}
    Let us choose $r\in(0,R)$ such that 
\begin{equation}\label{eq:13}
  r <\sqrt\frac{\lambda_1(D_1)}{\Lambda}, 
\end{equation} 
where $\lambda_1(D_1)$ is the 1-st eigenvalue of the Laplacian in the
unit disk $D_1$. According to Lemma \ref{lemRing} there exists
$\delta(r)>0$ such that, if $a<\delta(r)$, any eigenfunction
associated to $\lambda_N^a$ does not vanish in the closed ring
$C_{r,R}$.
  
  Let us assume that $0<a<\delta(r)$ and $a<r$ and let $u_N^a$ be an
  eigenfunction associated with $\lambda_N^a$. The proof relies on a
  topological analysis of $\mathcal N':=\mathcal N(u_N^a)\cap D_R$,
  inspired by previous work on nodal sets and minimal partitions (see
  \cite[Section 6]{BNH2015} and references therein). Lemma
  \ref{lemRing} implies that $\mathcal N'$ is compactly included in
  $D_r$.
 Theorem \ref{thmNodalLoc} implies that $\mathcal N'$ consists of a
 finite number of regular arcs connecting a finite number of singular
 points. In other words, $\mathcal N'$ is a regular planar graph. Let
 us denote by $V$ the set of vertices of $\mathcal N'$, by $b_1$ the
 number of its connected components and by $\mu$ the number of its
 faces.
 By face, we mean a connected component of $\RR^2\setminus\mathcal
 N'$. There is always one unbounded face, so $\mu \ge 1$. 
Furthermore, for all $v\in V$, we denote by $\nu(v)$ the degree of the
vertex $v$, that is to say the number of half-curves ending at $v$.
 Let us note that, according to Theorem \ref{thmNodalLoc}, both $a^-$
 and $a^+$
 belong to $V$ and have an odd degree, and any other vertex can only 
have an even degree. These quantities are related through Euler's formula for planar graphs:
\begin{equation}
\label{eq:Euler}
	\mu=b_1+\sum_{v\in V}\left(\frac{\nu(v)}{2}-1 \right)+1.
\end{equation}
For this classical formula, see for instance \cite[Theorems 1.1 and 9.5]{BM76}. Note that this reference treats the case of a connected graph. The generalization used here is easily obtained by linking the $b_1$ connected components of the graph with $b_1-1$ edges, in order to go back to the connected case.

Let us show by contradiction that $\mu =1$. If $\mu \ge 2$, there  exists a bounded face of the graph $\mathcal N'$, which is a nodal domain of $u_N^a$ entirely contained in
  $D_r$. Let us call it $\omega$. We denote by 
  $\lambda_k(\omega,a^-,a^+)$ the $k$-th eigenvalue of the
  operator $(i\nabla+\mathbf{A}_{a^-\!,a^+})^2$ in $\omega$, with
  homogeneous Dirichlet
  boundary condition on $\partial\omega$. Since $\omega$ is a nodal
  domain, we have that, for some $k(a)\in\N\setminus\{0\}$ depending on $a$,
\[
\lambda_N^a=\lambda_{k(a)}(\omega,a^-,a^+)\geq \lambda_{1}(\omega,a^-,a^+).
\]
By the diamagnetic inequality
\[
\lambda_{1}(\omega,a^-,a^+)\geq \lambda_{1}(\omega)
\]
where $\lambda_1(\omega)$ is the $1$-st eigenvalue of the Dirichlet Laplacian in
$\omega$. By domain monotonicity
\[
\lambda_1(\omega)\ge \lambda_1(D_r)=\frac{\lambda_1(D_1)}{r^2}.
\]
Hence we obtain that 
\begin{equation*}
  r\ge \sqrt\frac{\lambda_1(D_1)}{\lambda_N^a},
\end{equation*}
thus contradicting \eqref{eq:13}. We conclude that $\mu=1$.

Going back to Euler's formula \eqref{eq:Euler}, we obtain
\begin{equation}
\label{eq:EulerIneq}
	\sum_{v\in V}\left(\frac{\nu(v)}{2}-1 \right)=-b_1\le -1.
\end{equation}
According to Theorem \ref{thmNodalLoc}, we have $\nu(v)/2-1\ge -1/2$
if $v\in \{a^-,a^+\}$ and $\nu(v)/2-1\ge 1$ if
$v\in V\setminus \{a^-,a^+\}$. Inequality \eqref{eq:EulerIneq} can
therefore be satisfied only if $V=\{a^-,a^+\}$ and
$\nu(a^-)=\nu(a^+)=1$, that is to say if $\mathcal N'$ is a regular
arc connecting $a^-$ and $a^+$.  
\end{proof}

We are now in position to prove Theorem \ref{t:reduct-dirichl-lapl}.

\begin{proof}[Proof of Theorem \ref{t:reduct-dirichl-lapl}]
From Lemma \ref{l:single_curve} it follows that, for $a$ sufficiently
small,  there exists  a curve $K_a$ in $\mathcal{N}(u_N^a)$ connecting
$a^-$ and $a^+$ and (in view of Lemma \ref{lemRing}) concentrating at $0$, where 
$u_N^a$ is a
   magnetic-real eigenfunction associated with
   $\lambda_N^a$.

Let us write $\Omega'_a=\Omega\setminus K_a$. Since $K_a$ is contained
in $\mathcal{N}(u_N^a)$, we have that there exists
$k(a)\in\N\setminus\{0\}$ (depending on $a$) such that  
\begin{equation}\label{eqMinusC}
	\lambda_N^a=\lambda_{k(a)}({\Omega'_a},a^-,a^+),
\end{equation}
where $\lambda_{k(a)}({\Omega'_a},a^-,a^+)$ denotes the $k(a)$-th
eigenvalue of $H_{a^-\!,a^+}^{\Omega'_a}$.

Let us consider a closed path $\gamma$ in $\Omega'_a$.  By definition
of $\Omega'_a$, $\gamma$ does not meet $K_a$, which means that $K_a$
is contained in a connected component of $\R^2\setminus \gamma$. Since
the function $X\mapsto \mbox{Ind}_{\gamma}(X)$ is constant on all
connected components of $\R^2\setminus \gamma$, we have that
$\mbox{Ind}_{\gamma}(a^-)=\mbox{Ind}_{\gamma}(a^+)$. According to
\eqref{eqWinding}, this implies that
\begin{equation*}
  \frac{1}{2\pi}\oint_{\gamma}{\mathbf A}_{a^-\!,a^+}\cdot d\mathbf{s}=0.
\end{equation*}
In view Lemma \ref{lemEquivPot}, we conclude that
${\mathbf A}_{a^-\!,a^+}$ is gauge equivalent to $0$ in $\Omega'_a$  and hence
Proposition \ref{p:unieq} ensures that 
\begin{equation}
	\label{eqGaugeZero}
\lambda_{k(a)}({\Omega'_a},a^-,a^+)=\lambda_{k(a)}({\Omega'_a}).
\end{equation}
Combining \eqref{eqMinusC} and \eqref{eqGaugeZero} we obtain 
\begin{equation}\label{eq:14}
\lambda_N^a=\lambda_{k(a)}({\Omega'_a}).
\end{equation}
We observe that $a\mapsto k(a)$ stays bounded
as $a\to0^+$; indeed if, by contradiction,
$k(a_n)\to+\infty$ along some sequence
$a_n\to 0^+$, by \eqref{eq:14} we should have 
$\lambda_N^{a_n}=\lambda_{k(a_n)}({\Omega'_{a_n}})\geq
\lambda_{k(a_n)}(\Omega)\to+\infty$ thus contradicting \eqref{eq:15}.

Then, for any sequence $a_n\to 0^+$, there exists a subsequence $a_{n_j}$
such that $k(a_{n_j})\to k$ for some $k$; since $k(a)$ is integer-valued
we have that necessarily  $k(a_{n_j})= k\in\N\setminus\{0\}$ for $j$
sufficiently large. Hence \eqref{eq:14} yields
$\lambda_N^{a_{n_j}}=\lambda_{k}({\Omega\setminus K_{a_{n_j}}})$. It is
well known (see e.g. 
\cite[Theorem 1.2]{Courtois1995Holes})
  that $\lambda_{k}({\Omega\setminus K_{a_{n_j}}})\to
\lambda_{k}(\Omega)$ as $j\to +\infty$; hence, taking into account
\eqref{eq:16}, we conclude that $k=N$. Moreover, since the limit of
$k(a_{n_j})$ does not depend on the subsequence and $a\mapsto k(a)$ is
integer-valued, we conclude that $k(a)=N$ for all $a$ sufficiently
small, so that \eqref{eq:14} becomes 
\[
\lambda_N^a=\lambda_{N}({\Omega'_a})
\]
and the proof is complete.
\end{proof}

\section{Proof of Theorem \ref{t:main}}\label{sec:proof-theor-reft:m}

We are in position to complete the proof of  Theorem \ref{t:main}.

\begin{proof}[Proof of Theorem \ref{t:main}]
For $a>0$ small, let $K_a\subset\Omega$ be as in Theorem
\ref{t:reduct-dirichl-lapl}. We denote as
\[
d_a:=\mathop{\rm diam}K_a
\]
the diameter of $K_a$. From Theorem  \ref{t:one} it follows that 
\[
\lambda_{N}(\Omega\setminus K_a)-\lambda_{ N}= u_N^2(0)
\dfrac{2\pi}{|\log d_a|} +
o\bigg(\dfrac{1}{|\log d_a|}\bigg), \quad \text{as }a\to 0^+.
\]
Hence, in view of Theorem \ref{t:reduct-dirichl-lapl}, 
\begin{equation}\label{eq:17}
\lambda_{N}^a-\lambda_{ N}= u_N^2(0)
\dfrac{2\pi}{|\log d_a|} +
o\bigg(\dfrac{1}{|\log d_a|}\bigg), \quad \text{as }a\to 0^+.
\end{equation}
From \eqref{eq:17} and Proposition \ref{propUpper} it follows that,
for every $\tau\in(0,1)$, 
\begin{equation*}
\dfrac{1}{|\log d_a|}\big(1+
o(1)\big)\le \frac{1}{(1-\tau)|\log a|}
\left(1+o(1)\right)
\end{equation*}
and then 
\begin{equation}\label{eq:18}
\dfrac{|\log a|}{|\log d_a|}\le 
\frac{1}{(1-\tau)}
\big(1+o(1)\big), \quad\text{as $a\to0^+$}.
\end{equation}
On the other hand, since $a^-,a^+\in K_a$, we have that $d_a\geq 2a$
so that $|\log a|\geq|\log d_a|+\log2$
and 
\begin{equation}\label{eq:19}
\dfrac{|\log a|}{|\log d_a|}\ge 1+O\left(\frac1{|\log
    d_a|}\right)=1+o(1) , \quad\text{as $a\to0^+$}.
\end{equation}
Combining \eqref{eq:18} and \eqref{eq:19} we conclude that 
\[
1\leq \liminf_{a\to 0^+}\dfrac{|\log a|}{|\log d_a|}\leq
\limsup_{a\to 0^+}\dfrac{|\log a|}{|\log d_a|}\leq
\frac{1}{(1-\tau)}
\]
for every $\tau\in(0,1)$, and then, letting $\tau\to0^+$, we
obtain that 
\begin{equation}\label{eq:20}
\lim_{a\to 0^+}\dfrac{|\log a|}{|\log d_a|}=1.
\end{equation}
The conclusion then follows from \eqref{eq:17} and \eqref{eq:20}.
\end{proof}


\begin{thebibliography}{99}

\bibitem{AF} L. Abatangelo, V. Felli, {\it Sharp asymptotic estimates
    for eigenvalues of Aharonov-Bohm operators with varying poles},
Calc. Var. Partial
Differential Equations 54 (2015), no. 4, 3857--3903.

\bibitem{abatangelo2016leading}  
L.~Abatangelo, V.~Felli, {\it 
On the leading term of the eigenvalue variation for Aharonov-Bohm operators with a moving pole}, SIAM J. Math. Anal. 48 (2016), no. 4, 2843--2868.

\bibitem{AFHL-1} L. Abatangelo, V. Felli, L. Hillairet, C. L\'ena,
  {\it Spectral stability under removal of small capacity sets and
  applications to Aharonov-Bohm operators}, Preprint 2016, {\tt arXiv:1611.06750}.
  
\bibitem{abatangelo2016boundary} 
L. Abatangelo, V. Felli, B. Noris, M. Nys,
{\it Sharp boundary behavior of eigenvalues for Aharonov-Bohm
  operators with varying poles},  
Preprint 2016, {\tt arXiv:1605.09569}. 


\bibitem{AlzFleTak03}
B. Alziary, J. Fleckinger-Pell\'e, P. Tak\'a\v{c}, {\it Eigenfunctions
  and Hardy inequalities 
for a magnetic Schr\"odinger operator in $\R^2$}, Math. Methods
Appl. Sci. 26 (2003), no. 13, 1093--1136.

\bibitem{BM76} J. A. Bondy, U. S. R. Murty, {\it Graph {T}heory with
    {A}pplications\/}, American Elsevier Publishing Co., Inc., New
  York (1976).

\bibitem{BNH2011} V. Bonnaillie-No\"el, B. Helffer, {\it Numerical analysis of nodal sets for
eigenvalues of Aharonov-Bohm Hamiltonians on the square with
application to minimal partitions}, Exp. Math. 20 (2011), no. 3,
304--322.

 \bibitem{BNH2015} V. Bonnaillie-No\"el, B. Helffer, {\it Nodal and spectral minimal partitions -- The state of the art in 2015--}, Preprint 2015, {\tt arXiv:1506.07249}.

\bibitem{BNHHO2009} V. Bonnaillie-No\"el, B. Helffer,
  T. Hoffmann-Ostenhof, {\it  Aharonov-Bohm Hamiltonians,
    isospectrality and minimal partitions},
 J. Phys. A 42 (2009), no. 18, 185203, 20 pp.


\bibitem{BNNNT} V. Bonnaillie-No\"el, B. Noris, M. Nys, S. Terracini,
  {\it On the eigenvalues of Aharonov-Bohm operators with varying
    poles}, Analysis and PDE 7 (2014), no. 6, 1365--1395.
    

\bibitem{Courtois1995Holes}
G. Courtois, {\it Spectrum of manifolds with holes},
J. Funct. Anal. 134 (1995), no. 1, 194--221.


\bibitem{FFT} V. Felli, A. Ferrero, S. Terracini, {\it Asymptotic
    behavior of solutions to Schr\"odinger equations near an isolated
    singularity of the electromagnetic potential},
  J. Eur. Math. Soc. (JEMS) 13 (2011), no. 1, 119--174.
  

\bibitem{HHO13} B. Helffer, T. Hoffmann-Ostenhof, {\it On a magnetic
    characterization of spectral minimal partitions\/},
  J. Eur. Math. Soc. (JEMS) 15 (2013), no. 6, 2081--2092.


\bibitem{HHOO99} B. Helffer, M. Hoffmann-Ostenhof,
  T. Hoffmann-Ostenhof, M.P. Owen, {\it Nodal sets for groundstates of
    Schr{\"o}dinger operators with zero magnetic field in non-simply
    connected domains\/}, Comm. Math. Phys. 202 (1999), no. 3,
  629--649.

\bibitem{HHOT09}
B. Helffer, T. Hoffmann-Ostenhof, S. Terracini.
{\it Nodal domains and spectral minimal partitions\/},
  Ann. Inst. H. Poincar\'e Anal. Non Lin\'eaire 26 (2009), no. 1, 101--138.

\bibitem{LW99} A. Laptev, T. Weidl, {\it Hardy inequalities for
    magnetic Dirichlet forms\/}, Mathematical results in quantum
  mechanics (Prague, 1998), 299-305, Oper. Theory Adv. Appl., 108,
  Birkh\"{a}user, Basel, 1999.
  
\bibitem{lena} C. L\'{e}na, {\it Eigenvalues variations for
    Aharonov-Bohm operators\/}, Journal of Mathematical Physics 56,
  011502 (2015); doi: 10.1063/1.4905647.
  
\bibitem{Lei83} H. Leinfelder, {\it Gauge invariance of Schr\"odinger
    operators and related
 spectral properties\/}, J. Operator Theory 9 (1983), no. 1, 163--179.

\bibitem{NNT} B. Noris, M. Nys, S. Terracini, {\it On the eigenvalues
    of Aharonov-Bohm operators with varying poles: pole approaching
    the boundary of the domain}, 
Communications in Mathematical
  Physics 339 (2015), no. 3, 1101--1146.
 
\bibitem{NT} B. Noris, S. Terracini, {\it Nodal sets of magnetic
    Schr\"odinger operators of Aharonov-Bohm type and energy
    minimizing partitions}, Indiana University Mathematics Journal 59
  (2010), no. 4, 1361--1403.

\end{thebibliography}
\end{document}